\documentclass[a4paper, 11pt]{amsart} 
\usepackage[
  margin=1.9cm,
  includefoot,
  footskip=30pt,
]{geometry}
\usepackage{amsmath,amssymb,amscd}
\usepackage{amsfonts}
\usepackage[english]{babel}
\usepackage[leqno]{amsmath}
\usepackage{amssymb,amsthm}
\usepackage{mathrsfs} 
\usepackage{amscd}
\usepackage{enumerate}
\usepackage{epsfig}
\usepackage{relsize}
\usepackage{layout}
\usepackage[usenames,dvipsnames]{xcolor}
\usepackage[backref=page]{hyperref}
\usepackage{tikz}
\hypersetup{
 colorlinks,
 citecolor=Green,
 linkcolor=Red,
 urlcolor=Blue}
\usepackage[matrix,arrow,tips,curve]{xy}
\input{xy}
\xyoption{all}
\definecolor{darkgreen}{rgb}{0.0, 0.7, 0.0}
\definecolor{purple}{rgb}{0.5, 0.0, 0.5}
\definecolor{red}{rgb}{0.8, 0.2, 0.0}

\newtheorem{thm}{Theorem}[section]
\newtheorem{bthm}{Theorem}
\newtheorem{bcor}{Corollary}
\newtheorem{lemma}[thm]{Lemma}

\newtheorem{prop}[thm]{Proposition}

\numberwithin{equation}{section}
\theoremstyle{definition}
\newtheorem{defi}[thm]{Definition}

\theoremstyle{remark}
\newtheorem{remark}[thm]{Remark}

\newcommand{\C}{\mathbb{C}}

\def \Im{{\rm Im}}

\def \P{\mathbb{P}}

\def \F{\mathcal F}

\def \L{\mathcal L}
\def \E{\mathcal E}
\def \G{\mathcal G}

\def \U{\mathcal U}
\def\O{\mathcal O}

\def\M0{\mathcal M^0}

\DeclareMathOperator{\Proj}{{Proj}}
\DeclareMathOperator{\Sym}{{Sym}}

\def\B{\mathbf{B}}

\title[On partially ample Ulrich bundles]{On partially ample Ulrich bundles}

\author[A.F. Lopez, D. Raychaudhury]{Angelo Felice Lopez and Debaditya Raychaudhury}

\address{\hskip -.43cm Angelo Felice Lopez, Dipartimento di Matematica e Fisica, Universit\`a di Roma
Tre, Largo San Leonardo Murialdo 1, 00146, Roma, Italy. e-mail {\tt angelo.lopez@mat.uniroma3.it}}

\address{\hskip -.43cm Debaditya Raychaudhury, Department of Mathematics, University of Arizona, 617 N Santa Rita Ave., Tucson, AZ 85721, USA. email: {\tt draychaudhury@math.arizona.edu}}

\thanks{The first author is a member of the GNSAGA group of INdAM and was partially supported by PRIN ``Advances in Moduli Theory and Birational Classification''.}


\begin{document}

\begin{abstract} 
We characterize $q$-ample Ulrich bundles on a variety $X \subseteq \P^N$ with respect to $(q+1)$-dimensional linear spaces contained in $X$.
\end{abstract}

\maketitle

\section{Introduction}

Let $X \subseteq \P^N$ be a smooth variety of dimension $n \ge 1$. The study of positivity properties of vector bundles $\E$ on $X$ is a classical one. Starting with Hartshorne's pioneering paper \cite{h}, several positivity notions have been introduced, among which, perhaps, the most important one is ampleness. The latter amounts to say that the tautological line bundle $\O_{\P(\E)}(1)$ is ample. One possible weakening of this notion, so that some properties are maintained, is $q$-ampleness, that we now recall (see for example \cite{to} and references therein).

\begin{defi}
Let $q \ge 0$ and let $\L$ be a line bundle on a scheme $Y$. We say that $\L$ is {\it $q$-ample} if for every coherent sheaf $\F$ on $Y$, there exists an integer $m_0>0$ such that $H^i(\F(m\L))=0$ for $m \ge m_0$ and $i > q$. Let $\E$ be a vector bundle on $Y$. We say that $\E$ is {\it $q$-ample} if $\O_{\P(\E)}(1)$ is $q$-ample. 
\end{defi}

In this paper we are interested in studying the above notion for a special class of vector bundles, namely for Ulrich bundles, that is bundles $\E$ such that $H^i(\E(-p))=0$ for all $i \ge 0$ and $1 \le p \le n$. The importance of Ulrich bundles is well-known (see for example \cite{es, be, cmp} and references therein). Positivity properties of Ulrich bundles have been studied recently \cite{lo, lm, ls, lms1, lms2}. In particular, in \cite[Thm.~1]{ls}, we showed that an Ulrich bundle $\E$ is ample (that is $0$-ample) if and only if either $X$ does not contain lines or $\E_{|L}$ is ample on any line $L \subset X$. We prove here a generalization of this result.

\begin{bthm} 
\label{ulrq}

\hskip 3cm

Let $X \subset \P^N$ be a smooth variety of dimension $n \ge 1$. Let $\E$ be an Ulrich vector bundle and let $q \ge 0$ be an integer. Then the following are equivalent:
\begin{itemize} 
\item[(i)] $\E$ is $q$-ample;
\item[(ii)] either $X$ does not contain a linear space of dimension $q+1$, or $\E_{|M}$ does not have a trivial direct summand for every linear space $M \subseteq X$ of dimension $q+1$;
\item[(iii)] either $X$ does not contain a linear space of dimension $q+1$, or $h^0(\E^*_{|M})=0$ for every linear space $M \subseteq X$ of dimension $q+1$.
\end{itemize} 
\end{bthm}

We also have the following consequence.

\begin{bcor}
\label{ulrqc}
Let $\E$ be an Ulrich vector bundle on $X \subseteq \P^N$. Then:
\begin{itemize}
\item[(i)] $\E$ is $(n-1)$-ample if and only if $(X, \O_X(1), \E) \ne (\P^n, \O_{\P^n}(1), \O_{\P^n}^{\oplus r})$. 
\item[(ii)] If $n \ge 2, (X, \O_X(1), \E) \ne (\P^n, \O_{\P^n}(1), \O_{\P^n}^{\oplus r})$ and $\rho(X)=1$, then $\E$ is $(n-2)$-ample.
\end{itemize}
\end{bcor}

In recent years, positivity of vector bundles have been measured by augmented and restricted base loci (see for example \cite{bkkmsu, fm}). In the last section we will ask a question about augmented base loci of Ulrich bundles arising from the above theorem. 

\section{Notation}

Throughout the paper we work over the field $\C$ of complex numbers. A {\it variety} is by definition an integral separated scheme of finite type over $\C$. 

\section{Generalities on vector bundles}

In this section we collect some general facts about vector bundles and some notation that will be used later.

\begin{defi}
\label{not}

Let $\E$ be a rank $r$ vector bundle on $X$. We set $\P(\E) = \Proj(\Sym(\E))$ with projection map $\pi : \P(\E) \to X$ and tautological line bundle $\O_{\P(\E)}(1)$. If $\E$ is globally generated we define the map determined by $|\O_{\P(\E)}(1)|$ as
$$\varphi=\varphi_{\E} = \varphi_{\O_{\P(\E)}(1)} : \P(\E) \to \P H^0(\E)$$
and we set
$$\Pi_y = \pi(\varphi^{-1}(y)), y \in \varphi(\P(\E)) \ \hbox{and} \ P_x = \varphi(\P(\E_x)), x \in X.$$
We also define the map determined by $\E$ as
$$\Phi=\Phi_{\E} : X \to {\mathbb G}(r-1, \P H^0(\E))$$
given, for any $x \in X$, by $\Phi(x)=[P_x] \in {\mathbb G}(r-1, \P H^0(\E))$. 
\end{defi}

We record some simple but useful facts. The first one is essentially contained in \cite[Proof of Lemma 2.4, page 426]{ta}\footnote{See also {\url https://mathoverflow.net/questions/395472/trivial-subbundle-of-universal-bundle-on-the-grassmannian-mathbbgk-n}.}.

\begin{lemma}
\label{tango}
Let $V$ be a vector space and let $P \in \P V$ be a point. Let $Y \subset {\mathbb G}(k, \P V)$ be a subvariety such that, for every $y \in Y$, the corresponding $k$-plane contains $P$. If $\U$ is the universal subbundle of ${\mathbb G}(k, \P V)$,  then $\U_{|Y} \cong \O_Y \oplus \G$, for some rank $k$ vector bundle $\G$ on $Y$.
\end{lemma}
\begin{proof}
The assertion being obvious if $\dim V = 1$, we assume that $\dim V \ge 2$.
Let $P= \P V_1$, where $V_1 \subseteq V$ is $1$-dimensional and choose a splitting $V = V_1 \oplus V'$. We have a closed embedding $j: G(k-1,V') \hookrightarrow G(k,V)$ defined by $j([W])=[V_1 \oplus W]$, where $V_1 \oplus W \subset V$, or, equivalently, $j$ is the morphism associated to the vector bundle $\O_{G(k-1,V')} \oplus (\U')^*$, where $\U'$ is the universal subbundle of $G(k-1,V')$.  Let $G_P = \{[U] \in G(k,V) : P \in \P U\}$. Then $j$ defines an isomorphism $G(k-1,V') \cong G_P$, hence 
$$\U_{|G_P} \cong j^*\U \cong \O_{G(k-1,V')} \oplus \U' \cong \O_{G_P} \oplus \G'$$
for some rank $k$ vector bundle $\G'$ on $G_P$.
Since $Y \subseteq G_P$, we get that $\U_{|Y} \cong \O_Y \oplus \G$, where $\G=\G'_{|Y}$.
\end{proof}

\begin{lemma}
\label{inj}
Let $\E$ be a globally generated rank $r$ vector bundle on $X$.
\begin{itemize}
\item[(i)] For every $x \in X$ the restriction morphism $\varphi_{|\P(\E_x)} : \P(\E_x) \to P_x$
is an isomorphism onto a linear subspace of dimension $r-1$ in $\P H^0(\E)$.
\end{itemize} 
Now let $y \in \varphi(\P(\E))$. Then:
\begin{itemize}
\item[(ii)] $\pi_{|\varphi^{-1}(y)} : \varphi^{-1}(y) \to X$ is a closed embedding. 
\item[(iii)] $\E_{|\Pi_y} \cong \O_{\Pi_y} \oplus \G$, for some rank $r-1$ vector bundle $\G$ on $\Pi_y$. 
\end{itemize} 
\end{lemma}
\begin{proof}
To see (i), observe that we have $\P(\E_x) \cong \P^{r-1}$ and $\O_{\P(\E)}(1)_{| \P(\E_x)} \cong \O_{\P^{r-1}}(1)$. Let 
$$W = \Im \{H^0(\O_{\P(\E)}(1)) \to H^0(\O_{\P^{r-1}}(1))\}.$$
Being $\O_{\P(\E)}(1)$ globally generated, we have that so is $W$, hence $\dim W \ge r$. It follows that $W = H^0(\O_{\P^{r-1}}(1))$ and $\varphi_{|\P(\E_x)} = \varphi_{\O_{\P^{r-1}}(1)}$ is an isomorphism onto its image, which is then a linear subspace of dimension $r-1$ in $\P H^0(\E)$. This proves (i) and then (i) implies that $\pi$ and its differential are injective on the fibers of $\varphi$, proving (ii). As for (iii), set $M=\Pi_y$ and consider the globally generated rank $r$ vector bundle $\E_{|M}$ on $M$. Let 
$$U = \Im \{H^0(\O_{\P(\E)}(1)) \to H^0(\O_{\P(\E_{|M})}(1)\}$$
so that $\varphi_{|\P(\E_{|M})} = \varphi_U : \P(\E_{|M}) \to \P U$. Set $\Phi_M=\Phi_{\E_{|M}}, \varphi_M = \varphi_{\E_{|M}}$ and, for any $x \in M$, $P_{M,x}=\varphi_M(\P((\E_{|M})_x)$. We have a commutative diagram
$$\xymatrix{\P(\E_{|M}) \ar[dr]_{\varphi_U} \ar[r]^{\hskip -1cm \varphi_M} & \hskip .2cm \varphi_M(\P(\E_{|M})) \subset \P H^0(\E_{|M}) \hskip -.9cm \ar[d]^p \\ & \hskip 1cm \varphi_U(\P(\E_{|M})) \subset \P U}$$
where $p$ is a finite map. For any $x \in M$, there is a $z \in \varphi^{-1}(y)$ such that $x = \pi(z)$. Hence $z \in \P(\E_x)=\P((\E_{|M})_x)$ and therefore $y = \varphi(z) = \varphi_U(z)=p(\varphi_M(z))$, so that $\varphi_M(z) \in p^{-1}(y) \cap P_{M,x}$. Therefore each $(r-1)$-plane $P_{M,x}$ passes through one of the points of $p^{-1}(y)$. On the other hand, the family of these $(r-1)$-planes is just $\Phi_M(M) \subset {\mathbb G}(r-1, \P H^0(\E_{|M}))$, thus it is irreducible. Since $p^{-1}(y)$ is finite and the condition of passing through a point is closed, we deduce that there is a point $y_M \in \P H^0(\E_{|M})$ such that $y_M \in P_{M,x}$ for every $x \in M$. Set $Y = \Phi_M(M) \subset {\mathbb G}(r-1, \P H^0(\E_{|M}))$. It follows by Lemma \ref{tango} that $\U^*_{|Y} \cong \O_Y \oplus \G$, for some rank $r-1$ vector bundle $\G$ on $Y$. Since $\E_{|M}=\Phi_M^*\U^*$, this proves (iii). 
\end{proof}

\section{$q$-ample vector bundles}

We discuss some generalities on $q$-ample vector bundles.

\begin{defi}
Let $\E$ be a vector bundle on $X$. We set $q_{\rm min}(\E)= \min\{q \ge 0 : \E \ \hbox{is} \ q\hbox{-ample}\}$.
\end{defi}
The definition of $q_{\rm min}(\E)$ implies that $\E$ is $q$-ample if and only if $q \ge q_{\rm min}(\E)$.

\begin{remark}
\label{somm}
We have:
\begin{itemize}
\item [(i)] If $\E$ is a globally generated vector bundle on $X$, then $\E$ is $q$-ample if and only if $\dim F \le q$ for every fiber $F$ of $\varphi=\varphi_{\E} : \P(\E) \to \P H^0(\E)$.
\item [(ii)] If $\E$ is globally generated, then it is $n$-ample. Also $n+r-1-\nu(\E) \le q_{\rm min}(\E) \le n$, where $r$ is the rank of $\E$ and $\nu(\E)$ is the numerical dimension of $\O_{\P(\E)}(1)$. 
\end{itemize}
\end{remark} 
\begin{proof}
(i) is just \cite[Prop.~1.7]{s}. The first part of (ii) follows either by \cite[Prop.~1.7]{s} or by (i), since $\dim \varphi^{-1}(y) = \dim \Pi_y \le n$ for every $y \in \varphi(\P(\E))$. Thus $q_{\rm min}(\E) \le n$. Since $\E$ is $q_{\rm min}(\E)$-ample, for any fiber $F$ of $\varphi$, we have by (i) that
$n+r-1-\nu(\E) \le \dim F \le q_{\rm min}(\E)$. This proves (ii).
\end{proof}

We have the following characterization, which is a special case of \cite[Prop.~1.7]{s}.

\begin{prop} 
\label{genq}
Let $X$ be a smooth variety of dimension $n \ge 1$. Let $\E$ be a globally generated vector bundle on $X$ and let $q \ge 0$ be an integer. Then the following are equivalent:
\begin{itemize} 
\item[(i)] $\E$ is $q$-ample;
\item[(ii)] $\E_{|Z}$ does not have a trivial direct summand for every subvariety $Z \subseteq X$ of dimension $q+1$;
\item[(iii)] $h^0(\E^*_{|Z})=0$ for every subvariety $Z \subseteq X$ of dimension $q+1$.
\end{itemize} 
\end{prop}

\begin{proof}
The equivalence (ii)-(iii) follows by \cite[Lemma 3.9]{oa}. As for the equivalence (i)-(ii), assume first that $\E_{|Z}$ does not have a trivial direct summand for every subvariety $Z \subseteq X$ of dimension $q+1$. If $\E$ is not $q$-ample, there exists by Remark \ref{somm}(i) an $y \in \varphi(\P(\E))$ such that $\dim \varphi^{-1}(y) \ge q+1$. Set $M=\Pi_y$. By Lemma \ref{inj}(ii) we have that $M \cong \varphi^{-1}(y)$, hence $\dim M \ge q+1$. Also, Lemma \ref{inj}(iii) implies that $\E_{|M} \cong \O_M \oplus \G$, for some vector bundle $\G$ on $M$. But then, for any subvariety $Z \subseteq M$ with $\dim Z = q+1$, we have that $\E_{|Z} \cong \O_Z \oplus \G_{|Z}$, contradicting the hypothesis. Vice versa, assume that $\E$ is $q$-ample and let $Z \subseteq X$ be a subvariety of dimension $q+1$. If $\E_{|Z} \cong \O_Z \oplus \G$, for some vector bundle $\G$ on $Z$, then $Z \cong \P(\O_Z) \subseteq \P(\E_{|Z}) \subseteq \P(\E)$ and
$$\O_{\P(\E)}(1)_{|\P(\O_Z)} \cong \O_{\P(\E_{|Z})}(1)_{|\P(\O_Z)} \cong \O_{\P(\O_Z)}(1) \cong \O_Z$$
hence $\varphi(\P(\O_Z))$ is a point. Therefore $\varphi$ has a fiber of dimension at least $q+1$, contradicting Remark \ref{somm}(i).
\end{proof}

\section{Proofs of the main results}

In the case of Ulrich vector bundles, we can do better than Proposition \ref{genq}. 

\renewcommand{\proofname}{Proof of Theorem \ref{ulrq}}
\begin{proof}
Recall that $\E$ is globally generated since it is $0$-regular. The equivalence (ii)-(iii) follows by \cite[Lemma 3.9]{oa}. As for the equivalence (i)-(ii), assume first that $\E$ is $q$-ample. Then either $X$ does not contain a linear space of dimension $q+1$ or it follows by Proposition \ref{genq} that $\E_{|M}$ does not have a trivial direct summand for every linear space $M \subseteq X$ of dimension $q+1$. To see the converse, let $y \in \varphi(\P(\E))$ and let $\Pi_y = \pi(\varphi^{-1}(y))$, so that $\Pi_y \cong \varphi^{-1}(y)$ by Lemma \ref{inj}(ii). By \cite[Thm.~2]{ls} we have that $\Pi_y$ is a linear space contained in $X$. Now, if $X$ does not contain a linear space of dimension $q+1$, then $\dim \varphi^{-1}(y) = \dim \Pi_y \le q$ for every $y \in \varphi(\P(\E))$. Hence $\E$ is $q$-ample by Remark \ref{somm}(i). On the other hand, assume that $\E_{|M}$ does not have a trivial direct summand for every linear space $M \subseteq X$ of dimension $q+1$. If $\E$ is not $q$-ample, there exists by Remark \ref{somm}(i) an $y \in \varphi(\P(\E))$ such that $\dim \varphi^{-1}(y) \ge q+1$. Hence $\dim \Pi_y \ge q+1$, and picking a linear subspace $M \subseteq \Pi_y$ with $\dim M = q+1$, we get a contradiction by Lemma \ref{inj}(iii).
\end{proof}
\renewcommand{\proofname}{Proof}

We also have. 

\renewcommand{\proofname}{Proof of Corollary \ref{ulrqc}}
\begin{proof} 
First we prove (i). If $\E$ is not $(n-1)$-ample, then it follows by Theorem \ref{ulrq} that $X=\P^n$ and $\E \cong \O_X \oplus \G$, for some vector bundle $\G$ on $X$. But then $\O_X$ is Ulrich and therefore $(X, \O_X(1), \E) = (\P^n, \O_{\P^n}(1), \O_{\P^n}^{\oplus r})$ by \cite[Lemma 4.2]{aclr}(vi) and \cite[Prop.~2.1]{es} (or \cite[Thm.~2.3]{be}). On the other hand, if $(X, \O_X(1), \E) = (\P^n, \O_{\P^n}(1),  \O_{\P^n}^{\oplus r})$, then $\P(\E) \cong \P^{r-1}\times \P^n$ and $\varphi=\pi_1 : \P^{r-1}\times \P^n \to \P^{r-1}$ has $n$-dimensional fibers, hence $\O_{\P^n}^{\oplus r}$ is not $(n-1)$-ample by Remark \ref{somm}(i). This proves (i). As for (ii), using Theorem \ref{ulrq}, we just need to prove that $X \subset \P^N$ does not contain linear spaces of dimension $n-1$ unless $(X, \O_X(1), \E) = (\P^n, \O_{\P^n}(1),  \O_{\P^n}^{\oplus r})$. To this end, let $A$ be the ample generator of $N^1(X)$ and let $H \in |\O_X(1)|$, so that $H \equiv hA$. If $X$ contains a linear space $M$ of dimension $n-1$, then $M \equiv aA$ for some integer $a \ge 1$, and therefore 
$$1 = M H^{n-1} =ah^{n-1}A^n$$
hence $a=h=A^n=1$ and then $H^n=1$, so that $(X, \O_X(1), \E) = (\P^n, \O_{\P^n}(1),  \O_{\P^n}^{\oplus r})$ by \cite[Prop.~2.1]{es} (or \cite[Thm.~2.3]{be}). This proves (ii). 
\end{proof}  
\renewcommand{\proofname}{Proof}

\section{Augmented base loci of Ulrich bundles}

Given a vector bundle $\E$, it follows by \cite[Thm.~1.1]{bkkmsu} that $\B_+(\E) \ne \emptyset$ if and only if $\E$ is not ample if and only if $\E$ is not $0$-ample. More generally, given $q \ge 0$, we have by Proposition \ref{genq} that $\E$ is not $q$-ample if and only if there exists a subvariety $Z \subseteq X$ of dimension $q+1$ such that $\E_{|Z}$ has a trivial direct summand. For any such subvariety, we have that $Z \cong \P(\O_Z) \subseteq \P(\E)$ and, since $\O_{\P(\E)}(1)_{|\P(\O_Z)} = \O_{\P(\O_Z)}(1) \cong \O_Z$, it follows that $\P(\O_Z) \subseteq \B_+(\O_{\P(\E)}(1))$. If $\pi : \P(\E) \to X$ is the natural map, then \cite[Prop.~3.2]{bkkmsu} implies that 
$$Z = \pi(\P(\O_Z)) \subseteq \pi(\B_+(\O_{\P(\E)}(1)))=\B_+(\E).$$
It is well-known, using  for example \cite[Prop.~2.3]{bbp}, that one cannot expect, in general, that $\B_+(\E)$ is the union of all such $Z$'s, already in the case of line bundles.

Now assume that $\E$ is Ulrich and not ample. It follows by \cite[Thm.~1]{ls} that there is a line $L \subseteq X$ such that $\E_{|L}$ is not ample. It was recently proved by Buttinelli \cite[Thm.~2]{bu} that
$$\B_+(\E) = \bigcup\limits_{L}L$$
where $L$ runs among all lines contained in $X$ such that $\E_{|L}$ is not ample. Equivalently $L$ runs among all lines contained in $X$ such that $\E_{|L}$ has a trivial direct summand. This is the case $q=0$ of a more general question. In fact, when $\E$ is not $q$-ample, we have by Theorem \ref{ulrq} that there is a linear space $M \subseteq X$ of dimension $q+1$ such that $\E_{|M}$ has a trivial direct summand. As above, this implies that $M \subseteq \B_+(\E)$.
Question: is $\B_+(\E)$ the union of all such $M$'s?

\end{document}